\documentclass[12pt]{amsart}
\usepackage{latexsym,amssymb,amsmath}
\usepackage{amsmath,amsfonts}
\usepackage{epsfig}
\usepackage{graphicx}
\usepackage[letterpaper, margin=1.2in]{geometry} 
 
\usepackage{verbatim}
\newtheorem{theorem}{Theorem}[section]

\newtheorem{prop}[theorem]{Proposition}
\newtheorem{cor}[theorem]{Corollary}

\newtheorem{rmk}[theorem]{Remark}

\allowdisplaybreaks

\def\T{\mathbb T}

\def\N{{\mathbb N}}
\def\Q{{\mathbb Q}}
\def\R{{\mathbb R}}
\def\Z{{\mathbb Z}}

\def\la{\langle}
\def\ra{\rangle}

\def\ges{\gtrsim}

\def\les{\lesssim}
\def\1{{\bf 1}}
\def\oo{{\"o}}

\def\eqnn{\begin{eqnarray*}}
\def\eeqnn{\end{eqnarray*}}
\def\eqn{\begin{eqnarray}}
\def\eeqn{\end{eqnarray}}
\newcommand{\nc}{\newcommand}
\nc{\be}{\begin{equation}}
\nc{\ee}{\end{equation}}
\nc{\ba}{\begin{eqnarray}}
\nc{\ea}{\end{eqnarray}}
\nc{\eps}{\epsilon}
\def\prf{\begin{proof}}
\def\endprf{\end{proof}}
\begin{document}

\title[Fractal solutions of  dispersive PDE]{Fractal solutions of dispersive partial differential equations on the torus}

\author[Erdo\u{g}an and  Shakan]{M. B. Erdo\u{g}an and  G. Shakan}
\thanks{The first author is partially supported by NSF grant  DMS-1501041. The second author  was partially supported by NSF grant  DMS-1501982 and  would
like to thank Kevin Ford for financial support. }
 \thanks{The authors would like   to thank Luis Vega for useful comments on an earlier version of this manuscript and for pointing out  several important references. The authors also thank an anonymous referee for useful suggestions and references.}
\address{Department of Mathematics \\
University of Illinois \\
Urbana, IL 61801, U.S.A.}
\email{berdogan@illinois.edu}

\address{Department of Mathematics \\
University of Illinois \\
Urbana, IL 61801, U.S.A.}
\email{george.shakan@gmail.com}

\date{}

\begin{abstract}
We use exponential sums to study the fractal dimension of the graphs of solutions to linear dispersive PDE. Our techniques apply to Schr\"odinger, Airy, Boussinesq, the fractional Schr\"odinger, and the gravity and gravity--capillary water wave  equations. We also discuss applications to certain nonlinear dispersive equations. In particular,  we obtain bounds for the  dimension of the graph of the solution to cubic nonlinear Schr\"odinger   and Korteweg--de Vries equations along oblique lines in space--time.  
\end{abstract}

\maketitle

\tableofcontents

\section{Introduction}

 Consider a linear dispersive equation  of the form
\be \label{eq:linear}\left\{\begin{array}{l}
iq_t+L_\omega q=0,\,\,\,\, t\in \R, \,\,\,x\in\T:=\R/(2\pi\Z) ,\\
q(0,\cdot)=g(\cdot)\in L^2(\T). 
\end{array}\right.
\ee
Here  $L_\omega$ is a linear symmetric differential operator of the form
$$
\widehat{L_\omega q}(n)= \omega(n) \widehat{q}(n), \,\,\,\, n\in\Z,
$$
where $\omega:\Z\to \R$ is the dispersion relation, and $\widehat{q}(n)$ denotes the $n^{\text{th}}$ Fourier coefficient of $q$. The solution of \eqref{eq:linear} can be written as
\be\label{solution}
q(t,x)=e^{itL_\omega}g(x)=\sum_{n\in\Z} \widehat{g}(n) e^{it\omega(n)+inx}.
\ee
The key examples are the  linear Schr\"odinger equation, $iq_t+q_{xx}=0$, with the dispersion relation $\omega(n)=-n^2$,  and the Airy (or linear Korteweg--de Vries (KdV)) equation, $q_t+q_{xxx}=0$, with $\omega(n)=n^3$.  

In this paper, using exponential sum estimates, we  study the fractal dimension of the graphs of solutions of a large class of dispersive PDE. Various aspects of this problem have  been   considered by many authors,   e.g., \cite{berlew,O,mber, berklei,KR,R,bermar,O2,OC10,olv,chenolv1,ET1,ET2,OC13,chenolv,HV,cet,V,ETbook,olvernew,olvtsa}.  
One of our goals is to give a theoretical justification to some observations of Berry \cite{mber}, Chen--Olver \cite{chenolv1,chenolv} and   Olver--Sheils \cite{olvernew}.

The initial motivation for this question goes back to an 1836   optical experiment by  Talbot \cite{talbot}. Talbot studied monochromatic light passing through a diffraction grating and  observed that at a certain distance (the so--called Talbot distance)  the diffraction pattern reproduces the grating pattern. Moreover at each rational multiple of the Talbot distance the pattern appears to be  a finite linear combination of the grating pattern.  
 Berry and his collaborators   studied this phenomenon in a series of papers, e.g.,  \cite{mber, berklei,berlew, bermar}.
In particular, in \cite{berklei},    the linear Schr\"odinger evolution was utilized to model the Talbot effect where time represents distance to the grating. The authors proved that at rational times the solution is a linear combination of finitely many translates of the initial data with the coefficients being Gauss sums, also see \cite{mtay1,mtay2,olv,olvtsa}.   This phenomenon is often called {\it quantization} in the literature. In \cite{berklei}, the authors  also observed that the solution at irrational times has a fractal nowhere differentiable profile. 
In particular for step function initial data at rational times one observes a step function,  and a nowhere differentiable function with fractal dimension $\frac32$ at irrational times. In addition, in \cite{mber}, Berry  argued that there should be space slices whose time fractal dimension is $\frac74$ and diagonal slices with dimension $\frac54$. Finally, it was conjectured that this phenomenon should   occur even when there is a nonlinear perturbation.

The first mathematically rigorous work in this area is due to Oskolkov. In \cite[Proposition 14]{O}, he proved that for bounded variation  initial data the solution of any linear dispersive PDE on $\T$ with a polynomial dispersion relation, in particular  the linear Schr\"odinger and the  Airy equations, is a continuous function  of $x$ at irrational times.  Moreover, if in addition the initial data is continuous, then the solution is a continuous function of space and time.  In \cite{KR}, 
Kapitanski and Rodniaski  showed that the solution to the linear Schr\"odinger equation at irrational times belong  to a higher regularity Besov space than   at rational times. This   effect can not be observed  in the scale of Sobolev spaces  since the linear propagator is  unitary. In \cite{R}, using the result in \cite{KR}, Rodnianski  justified  Berry's conjecture for the linear Schr\"odinger evolution proving that for initial data in\footnote{Throughout the paper we use the notation $H^{r+}(\T)$ to denote the set of functions $\bigcup_{s>r} H^s(\T)$, and   $H^{r-}(\T):= \bigcap_{s<r} H^s(\T)$. In addition, we use $BV(\T)$ to denote the set of bounded variation functions on $\T$.} $BV(\T)\setminus H^{\frac12+}(\T)$,  the graph of the real and imaginary parts of  the solution  has fractal dimension $\frac32$ at almost every  time.  In \cite{O}, as well as in  \cite{KR,R}, the proof relies on the properties of   the following discrete Hilbert transform
\be\label{eq:H}
H(t,x)=p.v.\sum_{n\neq 0} \frac1n e^{in^2t+inx}.
\ee 
The fractal dimension claim  of \cite{R} follow from the observation that for almost every time $H(t,x)$ is   $C^{\frac12-}$ as a function of $x$.

Several studies in the literature focused on determining the fractal dimension of analogs of \eqref{eq:H}.  Riemann's proposed continuous but nowhere differentiable function, $$\phi(t) = \sum_{n \neq 0} \frac{e^{itn^2}}{n^2},$$ is essentially  the integral of the  fundamental solution to the Schr\"odinger equation, 
 $\sum_n e^{itn^2 + inx},$ 
along the vertical line $x =0$. 
In analogy  to quantization of the Talbot effect, it is well--known that $\phi$ is nondifferentiable except at certain rational values of $t/(2\pi)$  where the corresponding Gauss sum vanishes. In a remarkable paper \cite{Ja}, Jaffard established the multifractality of $\phi$ obtaining precise bounds for the Hausdorff dimension of the sets where $\phi$ is $C^\alpha$. If one instead fixes a time $t = t_0$ and integrates the fundamental solution horizontally, then one is lead to the study of $H(t_0,x)$.  The analysis of the fractal dimension of these graphs is more delicate; the only   known results are due to Oskolkov and Chakhiev \cite{OC13}. In particular,  \cite[Corollary 2]{OC13}  states that for almost every $x_0$, the function $H$  restricted to the vertical line $x=x_0$ is $C^{\frac18-}_t$, which implies the fractal dimension of the graph is at most $ \frac{15}8$ (see the proof of Corollary~\ref{cor:obliqueupper} below). They also obtained  bounds for the Hausdorff dimension of the exceptional set of times at which $H(t,x)$ fails to be $C^\alpha_x$ for a given $\alpha\in (0,\frac12)$.   One of the main goals of our work is to initiate the study of oblique lines  $k x + \ell t = c$ with $k , \ell \in \mathbb{Z},$  as was proposed in \cite{mber}, whilst using the Schr\"{o}dinger equation to model the Talbot effect.

In \cite{ChCo}, Chamizo and Cordoba considered variations of $\phi$ such as $$ \phi_{a,k}(t) := \sum_{n =1}^\infty \frac{e^{itn^k}}{n^a} , \ \ \  \ \ \frac{k+1}{2} \leq a \leq k + \frac12.$$ They proved that the fractal dimension of the graph is exactly $2+\frac{1 - 2a}{2k}$. We remark that $\phi_a(t)$ is a regularized analog of the solution of \eqref{eq:linear} with dispersion relation $\omega(n) = n^k$, along the vertical line $x = 0$, see Remark~\ref{rmk:ChCo} below for further discussion.

In \cite{ET1,ET2}, using smoothing estimates for the evolution, the first author and Tzirakis  extended Oskolkov's \cite{O} and Rodnianski's \cite{R} results to the  KdV  and the cubic nonlinear Schr\"odinger  (NLS) equations, in particular they proved that the dimension of the graph for NLS solutions is $\frac32$ under the conditions above. In \cite{cet,ETbook} the authors addressed Berry's conjecture on the  fractal dimension of the density function $|e^{it\partial_{xx}} g|^2$ and extended Rodnianski's result to higher order dispersive equations  with polynomial dispersion relation. In addition they studied applications to the vertex filament equation.

In this paper we first consider the case of polynomial dispersion relations with integer coefficients. In 
Section~\ref{sec:poly}, we obtain dimension bounds for the graph of solutions restricted to oblique\footnote{Our method  also applies to vertical lines $x=x_0$ in the polynomial case, see Remark~\ref{rmk:vert}.} lines in space--time  with rational slope:
\begin{theorem}\label{thm:obliq} Let $\omega$ be a polynomial of degree $d\geq 2$ with integer coefficients.   Let $g$ be  a non--constant  step function.    Then for all $r\in \Q$ and a.e.~$c$, the  function 
$$f(x)=e^{itL_\omega} g\big|_{t=c-rx}$$ 
is in $  C^{\alpha}_x$, for every $\alpha<\frac1{d(2^d+1)}$. Moreover,   the maximum dimension of the graph of real and imaginary parts of $f$  is in $[2-\frac1{2d}, 2-\frac1{d(2^d+1)}]$.  
\end{theorem}

In \cite{O2}, Oskolkov already studied the linear Schr\"odinger evolution on oblique lines. In our notation he proved that the restriction of the function $H$, \eqref{eq:H}, to the   lines $t=Nx+M\pi$ and $x=Nt+M\pi$, where $M$, $N$ are odd integers, is $C^{\frac14}$. This result can be extended to the linear Schr\"odinger evolution provided that the data is chosen carefully\footnote{Using \cite[Lemma 2]{O2} and following the proof of Corollary~\ref{cor:obliqueupper} below,  one can easily prove that for  any  odd $N\geq 3$, if $g$ is a non--constant step function on the torus with steps only at the points $\{\frac{ M\pi}N:M\text{ odd integer in } (-N,N]\}$, then the the dimension of the graph of the function $e^{it\partial_{xx}}g\big|_{t=Nx}$  is at most  $\frac74$.   Note that in this result one needs to choose the initial data depending on the oblique line. We also note that our lower bound (Proposition~\ref{prop:obliquelower}) applies to this case, and hence the  dimension is exactly $\frac74$. }  depending on the line.  

The lower bound in Theorem~\ref{thm:obliq} follows from an    $L^4$ argument which applies to more general functions, see Proposition~\ref{prop:obliquelower}. A reason for expecting a larger dimension than $\frac32$ is the observation that $\widehat g(n)$ is essentially the Fourier coefficient  at frequency $\sim n^d$ for the diagonal slices, which drops the Sobolev index by a factor of $d$. The lower bound for the dimension implies that the restriction of the solution to such lines cannot be smoother than $C^{\frac1{2d}}$.   However, for $r\in\Q$ and a.e.~$c$, we prove that the restriction of the solution to the lines $t=c- r x$ is $C^{\alpha_d }$ for some $\alpha_d>0$, which yields
  an upper bound for the fractal dimension. To do this we obtain    exponential sum estimates such as (when $d=2$) 
$$\sup_{x} \big| \sum_{n\sim N} e^{i nx + in^2 (c-rx)}\big| \lesssim_{\epsilon} N^{\frac45 + \epsilon},$$ for $r\in\Q$ and a.e.~$c \in \mathbb{R}$.  
 Furthermore, we 
extend the dimension bounds to cubic NLS  and KdV   evolutions using nonlinear smoothing estimates from \cite{ET1,egtz1}. More precisely, we prove 
\begin{theorem}\label{thm:NLSoblique}
Consider the Wick ordered cubic NLS equation\footnote{We consider the Wick ordered case for simplicity, the result also applies to the regular NLS equation.} 
\be\label{NLSwick}
\left\{\begin{array}{l}
iu_t+u_{xx}\pm |u|^2u \mp Pu=0,\,\,\,\, t\in \R, x\in\T,\\
u(0,\cdot)=g(\cdot), 
\end{array}\right.
\ee
where  $g$ is a non--constant step function and $P=\frac1\pi\|g\|_{L^2(\T)}^2$.  
 Fix  $k ,  \ell\in \N$   with $(k,\ell) = 1$. 
 For $c\in \R$,  let $F_c(x)=u(c-\frac{k}\ell x,x)$, $x\in [0,2\pi\ell]$. Then 
 for a.e.~$c$, the function $F_c$ is in $ C^{\frac1{10}-}$ and we  have $D_c\in[\frac74,\frac{19}{10}]$, where $D_c$ is  the maximum dimension of the graphs of real and imaginary parts of $F_c$.  
\end{theorem}
The statement follows from the $d=2$ case of Theorem~\ref{thm:obliq}  for the linear group, and the following implication of a nonlinear  smoothing estimate in \cite{egtz1} which is of independent interest: For initial data $g\in H^{\frac12-}(\T)$, the solution of \eqref{NLSwick} satisfies locally in time:
\be\label{calphasmooth}
u-e^{it\partial_{xx}}g\in C^{\frac12-}_{t,x}.
\ee
Recall that, because of quantization, the linear evolution cannot even be a continuous function of $x$ at rational times  when $g$ is a step function. We also note that a lower bound for the dimension can be obtained for more general data $g$, see Remark~\ref{rmk:obliqlower}.

In the case of     the KdV equation, 
\be\label{KdV}
\left\{\begin{array}{l}
u_t+u_{xxx}+uu_x=0,\,\,  x\in\T, \,\, t\in \R,\\
u(0,\cdot)=g(\cdot), 
\end{array}\right. 
\ee 
or the Airy equation, 
a statement similar to Theorem~\ref{thm:NLSoblique} is valid: 
\begin{theorem}\label{thm:kdvoblique}
Let $g$ be a non--constant, mean--zero, and  real valued step function on the torus. Let $u(t,x)$ solve the Airy equation or the KdV equation \eqref{KdV} with data $g$.   
 Fix  $k ,  \ell\in \N$   with $(k,\ell) = 1$. 
 For $c\in \R$,  let $F_c(x)=u(c-\frac{k}\ell x,x)$, $x\in [0,2\pi\ell]$. Then 
 for a.e.~$c$,   $F_c \in C^{\frac1{27}-}$ and the dimension of the graph of $F_c$ is in $  [\frac{11}6,\frac{53}{27}]$.
\end{theorem}
In addition, analogous to \eqref{calphasmooth}, the smoothing result of \cite{ET1} implies that the nonlinear part of the KdV evolution for $H^{\frac12-}$ data is in $C^{\frac13-}_{t,x}$. 

In Section~\ref{sec:poly_space} we  obtain dimension estimates for space slices in the case of polynomial dispersion relation improving some results of \cite{cet}. In particular, concerning the Airy and KdV evolutions, the authors in \cite{cet} obtained  the dimension bounds $D\in[\frac54,\frac74]$. We improve the lower bound:
\begin{theorem}\label{thm:kdv3/2} Let $g\in BV(\T)\setminus H^{\frac12+}(\T)$ be real valued. Then for a.e.   $ t $, the dimension of the graph of both $e^{-t\partial_{xxx}}g$ and  the solution $u(t,\cdot)$ of \eqref{KdV} is at least $  \frac32$. 
\end{theorem}
 
In Section~\ref{sec:nonpoly}, we  obtain dimension  estimates for the graphs of   equations with non--polynomial dispersion relations.
Olver and his collaborators  \cite{olv,chenolv1,chenolv,olvernew}  provided numerical simulations of the Talbot effect for a large class of dispersive equations. In the case of polynomial dispersion, they numerically confirmed the rational/irrational dichotomy discussed above. An interesting  question that the authors raised is the appearance of such phenomena in the case of non--polynomial dispersion relations, $\omega(n)$.   Important examples are fractional Sch\"odinger/Airy type equations ($\omega(n)=|n|^\alpha, \alpha>0, \alpha\not\in \N$), 
Boussinesq equation   ($\omega(n)=\sqrt{n^2+n^4}$),  the gravity water wave equation ($\omega(n)=\sqrt{n \tanh(n)}$), and the gravity--capillary wave equation ($\omega(n)=\sqrt{(n+n^3) \tanh(n)}$). In Section~\ref{sec:nonpoly}, we obtain dimension bounds for the solution  of each of these equations using exponential sum estimates from \cite{Vau,Gr,Iw,Bo,He}. For example using van der Corput bounds we obtain\footnote{As in the case of the cubic NLS equation discussed above the dimension bounds for $\omega(n)=|n|^\alpha$, $\alpha\in (1,2)$ can be extended to the cubic fractional NLS equation using the smoothing bounds in \cite{det,egtz1}, see Theorem~\ref{thm:fracNLS} below.} 
\begin{theorem}\label{thm:nalpha}
For $\alpha\in (0,2)\setminus\{1\}$ define 
$$\beta =\left\{\begin{array}{ll}
\frac\alpha2,&  \alpha \in(0,1),\\
 1-  \frac\alpha2, & \alpha\in(1,\frac32],\\
  \frac12 - \frac\alpha6, &  \alpha\in( \frac32,2).
  \end{array}\right.
$$   For any $g\in BV(\T)$ and for each $t\neq 0$, the linear fractional Schr\"odinger evolution satisfies: $e^{it(-\Delta)^{\frac\alpha2} } g\in C^\beta_x$. In particular\footnote{We define $D_t(\omega,g)$ as the maximum  dimension  of the graphs of real and imaginary parts of $e^{itL_\omega} g$ as  functions of $x$.},   $D_t(\omega,g)\leq 2-\beta$. If in addition $g\not\in H^{\frac12+}(\T)$, then for each $t\neq 0$, $D_t(\omega,g)\geq 1+\beta$.
\end{theorem}
This theorem  indicates  that the rational/irrational dichotomy discussed above is not valid for fractional Schr\"odinger equation. In particular, the quantization fails   as the solution is a continuous function for each $t\neq 0$. This statement is not surprising   since $\{t|n|^\alpha (\text{mod }1): n\in \Z\}$ is   uniformly distributed on $[0,1]$ for noninteger  $\alpha>0$ and  for  $t\neq 0$, whereas  the set $\{tn^2 (\text{mod }1): n\in \Z\}$ is finite  for rational $t$. 
On the other hand, we observe some dependence on the algebraic nature of time.  For example, in the case  $\omega(n)= |n|^{\frac32}$, we prove that $D_t(\omega,g) \in [\frac{11}{8},\frac{13}8]$  for a.e.~$t$ (characterized by Khinchin--L\'evy numbers), see Theorem~\ref{t3/2} and  Remark~\ref{rmk:32}. This improves the bound $D_t(\omega,g) \in[\frac{5}{4},\frac{7}4]$ that Theorem~\ref{thm:nalpha} implies,  which is valid for each nonzero $t$.  However, it is not clear to us whether this dependence on the algebraic nature of $t$ is an artifact of our methods. For further discussion see Remark~\ref{rmk:32}.
Finally, we note that for some irrational values of $\alpha$  the bounds given by Theorem~\ref{thm:nalpha} can be improved by utilizing recent developments on the exponent pair conjecture, see Section~\ref{sec:ep}.
  
\subsection{Notation}
\begin{itemize}
\item For a family of Banach spaces, $B^s$, with regularity index $s$, we define
$$B^{s+}=\bigcup_{r>s}B^r  \text{ and } B^{s-}=\bigcap_{r<s}B^r.
$$
\item  Recall that the fractal  (also known as upper Minkowski or upper box--counting) dimension, $\overline{\text{dim}}(E)$, of a bounded set $E$ is given by $$\limsup_{\epsilon\to 0}\frac{\log({\mathcal N}(E,\epsilon))}{\log(\frac1\epsilon)},
$$
where ${\mathcal N}(E,\epsilon)$ is the minimum number of $\epsilon$--balls required to cover $E$. 
\item
We define $D_t(\omega,g)$ as the maximum  dimension  of the graphs of real and imaginary parts of $e^{itL_\omega} g$ as  functions of $x$. 

\item We say   $\beta \notin \mathbb{Q}$ is  Khinchin--L\'{e}vy  if there is a strictly  increasing   sequence   $\{q_n\}_{n\in\N}$ of natural numbers such that for every $\epsilon > 0$,  $q_{n+1} \leq q_n^{1+\epsilon}$ for all $n>N_\epsilon$, and if for each $n$ there is an $a_n$ coprime to $q_n$ with $|\beta - \frac{a_n}{q_n}| \leq \frac1{q_n^2}$.  By a theorem of  Khinchin \cite{khin} and L\'{e}vy \cite{levy}   a.e.~$\beta\in\R$ is Khinchin--L\'{e}vy.
\item We write $A\les B$ to denote $A=O(B)$, and $A\sim B$ to denote $A\les B$ and $B\les A$. We also write 
$A\les B^{c+}$ to denote for any $\epsilon >0$, $A\les B^{c+\epsilon}$ with implicit constant depending on $\epsilon$. Similarly, we will use $A\ges B^{c-}$.
\end{itemize}

\section{Overview and analytic reduction to exponential sums}

We first give a sketch of the basic idea behind the proofs. We consider a solution to a dispersive PDE \eqref{eq:linear} on $\T$, with dispersion relation $\omega$. The key example is $iq_t + q_{xx} = 0$ with initial data $\chi_{[0,\pi]}$. For a line $\mathcal{L} \subset \T^2$, we are interested in the  fractal  dimension of the graph of (real or imaginary parts of)  $q|_{\mathcal{L}} (t,x)$, which will typically lie in the open interval $(1,2)$. Our goal is to establish this rigorously and provide upper and lower bounds   for the  fractal  dimension. In the above example, it was established in \cite{R} that the  fractal  dimension is exactly $\frac32$ for a.e.~$t$. 

We modify the approach used in \cite{R,ET2,cet,ETbook}, also see \cite{O,KR}.   Recall the simple facts that the graph  of a  $C^{\gamma}(\T)$ function has fractal  dimension $\leq 2-\gamma$, and that for $f \in C^{\gamma}(\T)$, $|\widehat{f}(n)| \lesssim |n|^{-\gamma}$. The theory of Besov spaces (see the discussion around Theorems \ref{thm:DJ} and \ref{DJcor} below) allows one to  partially  reverse this implication, as well as provide lower bounds for the fractal dimension for $u|_{\mathcal{L}} $ as long as we can appropriately estimate $$\sum_{N \leq n < 2 N } e^{it \omega (n) + inx}$$ in various spaces for $(t,x) \in \mathcal{L}$.   We seek to improve upon the trivial bound $ N$, which  gives a nontrivial bound for the fractal dimension.

To make this idea more precise we utilize  the Littlewood--Paley projections
$$
  P_N\big(e^{itL_\omega}g\big)= \sum_{ N\leq | n| <2 N  } \widehat{g}(n) e^{it\omega(n)+inx}.
$$ 
We define the Besov space\footnote{In fact one needs to consider Littlewood-Paley projections with smooth cutoffs in this definition for certain statements in the paper, in particular in Theorem~\ref{thm:DJ}. We will ignore this issue as the upper bounds we have can easily be extended to the smooth case since the smooth projections are  uniformly bounded   in $L^p$ spaces  and since the lower bounds follow by interpolation.},   $B^s_{p,\infty}$, by the norm: 
$$\|f\|_{B^\gamma_{p,\infty}}:=\sup_{N\geq 1, \text{ dyadic}}N^{\gamma} \|P_N f\|_{L^p},\,\,\,\,\,1\leq p\leq \infty.$$
Recall that for $0<\gamma<1$, $C^\gamma(\T)$ coincides with $B^\gamma_{\infty,\infty}(\T)$, see, e.g., \cite{tri}, and that  if $f:\T\to\R$ is in $C^\gamma$, then
the graph of $f$ has fractal  dimension $D\leq 2-\gamma$. We also have the following theorem of Deliu and Jawerth \cite{DJ}, also see \cite[Theorem 2.24]{ETbook}:
\begin{theorem}\label{thm:DJ}\cite{DJ} Fix $\gamma\in [0,1]$. The graph of a continuous function $f:\T\to\R$  has fractal  dimension $D \geq 2-\gamma$ provided that $f\not\in B^{\gamma+}_{1,\infty}$.
\end{theorem}
 As a corollary of these statements we have the following theorem. A similar method was used in \cite{KR,R,ET1,cet}; the new ingredient here is the  observation that one can obtain lower bounds for the dimension using $L^q$ estimates.  This makes Strichartz estimates useful for finding lower bounds on the dimension, see Theorem~\ref{thm:strichloss}.

\begin{theorem}\label{DJcor}
Let $g\in BV(\T)$, and let  $L_\omega$ be a symmetric differential operator with dispersion relation $\omega$. Define   
$$H_{N,w}^\pm(x,t)= \Big|\sum_{ N\leq n <2 N }  e^{it\omega(\pm n)\pm i nx}\Big|.
$$
i) Fix $t$ and assume that for some\footnote{The range of $\gamma$ is determined by observing that the $L^2$ norm of $H_{N,w}^\pm(x,t)$ is $\sqrt{N}$, and the $L^\infty $ norm is $\leq  N$.}  $\gamma \in (0,1/2]$, we have an a priori bound of the form
\be\label{expboundinfty}
\|H_{N,w}^\pm(x,t)\|_{L^\infty_x}\les N^{1-\gamma+}, \,\,\,N\in\N.
\ee
Then $e^{itL_\omega}g\in C^{\gamma-}$ and hence $D_t(\omega,g)\leq 2-\gamma$.  \\
ii) Fix $t$  and assume that for some $q\in(2,\infty]$ and $\gamma \in [0,1/2]$, we have an a priori bound of the form
\be\label{expboundp}
\|H_{N,w}^\pm(x,t)\|_{L^q_x}\les  N^{1-\gamma+},\,\,\,N\in\N.
\ee
 If in addition  $e^{itL_\omega}g$ is continuous in $x$ and  $g\not\in H^{r+}(\T)$ for some $r\geq \frac12$, then 
 $$D_t(\omega,g)\geq  2-\frac{2r-\gamma q^\prime}{2-q^\prime}, \ \ \ \  q' : = \frac{q}{q-1}.$$
\end{theorem}
\begin{rmk} One way to establish the continuity hypothesis needed in part  ii) is to show  that part i) holds for some $\gamma > 0$. When the dispersion relation is a polynomial, it follows from \cite[Proposition 14]{O} for irrational values of $\frac{t}{2\pi}$.
\end{rmk}
\begin{proof}
For part i), by Theorem~\ref{thm:DJ}, it suffices to prove that  
$$
\big\|P_N\big(e^{itL_\omega}g\big)\big\|_{L^\infty_x}=\Big\| \sum_{ N\leq | n| <2 N } \widehat{g}(n) e^{it\omega(n)+inx}\Big\|_{L^\infty_x}\les N^{-\gamma+}.
$$
Since $g\in BV(\T)$, we can write
$$
P_N\big(e^{itL_\omega}g\big)=(dg)*\widetilde H_N
$$
where
$$
\widetilde H_N=\sum_{ N\leq |n| <2 N } \frac1n e^{it\omega(n)+inx}.
$$
Using the reverse Bernstein inequality\footnote{Also known as Bohr's Theorem, it states that for any trigonometric polynomial $P$ whose frequency support does not intersect $[-K,K]$ and for any  $1\leq q\leq \infty$ we have $\|P^\prime\|_{L^q}\geq C  K \|P\|_{L^q}$, where the constant $C$ is independent of $P, K$, and $q$. See, e.g., \cite{katz}.} and the a priori bounds \eqref{expboundinfty}, we obtain $\|\widetilde H_N\|_{L^\infty}\les  N^{-\gamma+}$. Since $dg$ is a finite measure, we conclude that  $e^{itL_\omega} g \in B^{\gamma-}_{\infty,\infty}$, which yields part i). 

The proof of   part ii) is similar. Using the a priori bounds in \eqref{expboundp} and the reverse Bernstein inequality we have  $\|\widetilde H_N\|_{L^q}\les N^{-\gamma+}$, which leads to
$$
\big\|P_N\big(e^{itL_\omega}g\big)\big\|_{L^q_x}\les N^{-\gamma+}.
$$ 
On the other hand, since $g\not \in H^{r+}(\T)$,  for any 
$s>r$
$$
\sup_N N^s\big\|P_N\big(e^{itL_\omega}g\big)\big\|_{L^2_x} =\sup_N N^s\big\|P_Ng\big\|_{L^2_x} =\infty.
$$
By interpolation we  conclude that
$$
\sup_N  N^\beta \big\| P_N\big(e^{itL_\omega}g\big)\big\|_{L^1_x} =\infty,\,\,\text{ for } \beta>\frac{2r-\gamma q^\prime}{2-q^\prime}.
$$  
Since $e^{itL_\omega}g$ is continuous in $x$, by Theorem~\ref{thm:DJ} we  conclude that the dimension of the graph of real or imaginary parts is $\geq 2-\frac{2r-\gamma q^\prime}{2-q^\prime}$. 
\end{proof}

\begin{rmk}\label{rmk:ChCo} An argument similar to the one given in Theorem~\ref{DJcor} was used by Chamizo and Cordoba   in \cite[Theorem~3.1]{ChCo}, where they studied 
$$\phi_{a , k}(t) = \sum_{n=1}^{\infty} \frac{c_n e^{in^k t}}{n^{a}}, \ \ \ (k+1)/2 \leq a \leq k + 1/2,\,\,\,c_n\approx 1.$$
For this range of $a$, they proved that the fractal dimension of the graph of $\phi_{a,k}$ is precisely 
$$2+\frac{1 - 2a}{2k}.$$ 
Instead of Besov spaces, their result relies on careful estimates on the function  $f(t):=\phi_{a, k}(t+h)-\phi_{a, k}(t)$. For the upper bound they utilize the mean value theorem and a large sieve inequality. 
For the lower bound they  use an $L^4-L^1$ interpolation argument similar to the one given in Theorem~\ref{DJcor} and Proposition~\ref{prop:obliquelower} below. However, the required bounds  hold under the regularity assumption $ a\geq (k+1)/2$.  Our use of Besov spaces and the fact that we work with a.e. line instead of $x=0$ allow us to study the Talbot effect with less regularity, in particular to establish  lower and upper bounds when $a =1$, see Proposition~\ref{prop:obliquelower},    
Corollary~\ref{cor:obliqueupper}, and Remark~\ref{rmk:vert}. 
\end{rmk}

\section{Equations with polynomial dispersion relations}\label{sec:poly}

In this section we consider the case when $\omega(n)$ is a polynomial in $n$ of degree $d$ with integer coefficients. 
First we address the dimension of graphs on oblique and vertical lines. Then, we consider the dimension on space slices.
 
\subsection{Results on oblique  lines} \label{sec:oblique}
In this section we consider the restriction of solutions to oblique lines in space time with rational slope and to vertical lines. We note that Theorem~\ref{thm:obliq} follows from Proposition~\ref{prop:obliquelower} and Corollary~\ref{cor:obliqueupper} below.

 Given   $c \in \R$, $k ,  \ell\in \N$   with $(k,\ell) = 1$,  and initial data  $g$, let 
\be\label{fgckl}
f_{g,c,k,l}(x)=e^{ itL_\omega} g\Big|_{ t=c-\frac{k}\ell x} =  \sum_{n} \widehat{g}(n) e^{i (c-\frac{k}\ell x  )\omega(n)+i nx}.
\ee
\begin{prop} \label{prop:obliquelower} Let $g\in BV(\T)\setminus H^{\frac12+}(\T)$.  Fix $c \in \R$, $k ,  \ell\in \N$   with $(k,\ell) = 1$. Assume that the  $2\pi\ell$-periodic function  $f_{g,c,k,l} $ defined above  is continuous. Then the dimension of the graph of real or imaginary parts of $f_{g,c,k,l} $ is $\geq 2-\frac1{2d} $ for each $c$ provided that $k\neq 1$. For $k=1$ the dimension is $\geq2-\frac1{2d}$ for a.e.~$c$. 
\end{prop}
\begin{rmk}\label{rmk:obliqlower} The continuity assumption holds for each $c,k,l$ in the case $g$ is continuous and bounded variation by a theorem of  Oskolkov \cite[Proposition 14]{O}. It also holds for a.e.~$c$ if $g$ is a step function by Corollary~\ref{cor:obliqueupper} below. A simple variation of the proof below yields that for $g\in BV(\T)$ if $r_0=\sup\{r:g\in H^r(\T)\} \in (\frac12,\frac34)$, then  
the dimension of the graph of real or imaginary parts of $f_{g,c,k,l} $ is $\geq 2-\frac{3r_0-1}{d} $ for each $c\in\R$ provided that $k\neq 1$. For $k=1$ the dimension is $\geq2-\frac{3r_0-1}{d}$ for a.e.~$c\in\R$. In this case the continuity assumption follows from Sobolev embedding and the continuity of the propagator in $H^r(\T)$. 
\end{rmk}
\begin{proof}[Proof of Proposition~\ref{prop:obliquelower}] 
Let 
$$F(x):=f_{g,c,k,\ell}(x\ell)=\sum_{n} \widehat{g}(n) e^{ic \omega(n)+x h(n)}  ,\,\,\,\,\,h(n):=\ell n - k\omega(n).$$
Note that $F$ is a $2\pi$-periodic  function and the series on the right converges in $L^2(\T)$. We have 
\be\label{eq:Fcoef}
\widehat F(m) =\sum_{n\in h^{-1}(m)}\widehat g(n) e^{i\omega(n)c}
=e^{-i\frac{mc}k}\sum_{n\in h^{-1}(m)}\widehat g(n) e^{i \frac{\ell n  c}k}.
\ee
 
By Theorem \ref{thm:DJ}, it is enough to show that $\|P_NF\|_{L^1(\T)}\gtrsim N^{-\frac1{2d}-}$ for infinitely many $N$. 

Since $g\in BV(\T)$, we have $|\widehat g(n)|\lesssim \frac1{\la n\ra}$, which leads to (for sufficiently large $N$ depending on $k, \ell$ and the coefficients of $\omega$)
\begin{align*}
 \|P_NF\|_{L^4(\T)}^4  &  \lesssim  \sum_{n_1,n_2,n_3,n_4: |n_i|\sim N^{\frac1d}} \frac1{N^{\frac4d}} \Big|\int_0^{2\pi}  e^{i(h(n_1)-h(n_2)+h(n_3)-h(n_4)  ) x} dx\Big|\\
&\lesssim N^{-\frac4d}    \# \big\{(n_1,n_2,n_3,n_4): |n_i|\sim N^{\frac1d}, h(n_1)+h(n_3)=h(n_2)+h(n_4)  \big\} \\ 
&\lesssim  N^{-\frac{2 }d+ }. 
\end{align*}
We explain the last inequality. First note that $h(n)-h(m)$ is divisible by $n-m$. Therefore, the condition in the set definition  can be rewritten as
$$ h(n_1)-h(n_2) =(n_4-n_3)\frac{h(n_4)-h(n_3)}{n_4-n_3}.$$
If $h(n_1) = h(n_2)$ then there are $\lesssim N^{\frac1d}$ choices for $n_1$ and $n_3$ and for each fixed $n_1$ and $n_3$, there are $O(d)$ choices for $n_2$ and $n_4$ as each is a 
root of a polynomial of degree $  d$. So we may suppose $h(n_1) \neq h( n_2)$. For fixed $n_1$ and $n_2$, there are at most $O(N^{0+})$ choices for $a = n_3 - n_4$ by the divisor bound. But $$\frac{h(n_4)-h(n_4+a)}{a},$$ is a non-constant polynomial of degree $d-1 \geq 1$  and so there are at most $O(d)$ choices for $n_4$ once $n_1 , n_2$, and $a$ are fixed. Then $n_3$ is fixed.

Using the bound above, we conclude that 
$$
\|P_NF\|_{L^4(\T)}\lesssim N^{-\frac{1}{2d}+}
$$

We now claim that for any $c\in \R$,  $F\not \in H^{\frac{1}{2d}+}(\T)$  provided that $k\neq 1$.   For $k= 1$ the claim holds for a.e.~$c$. This claim yields the proposition since $F\not \in H^{\frac1{2d}+}(\T)$ implies  that for infinitely many $N$
$$
\|P_NF\|_{L^2(\T)} \gtrsim N^{-\frac{1}{2d} - }.
$$
The required lower bound in $L^1$ follows from this and the $L^4$ upper bound above by interpolation. 

To see that the claim holds  for  $ k \neq 1$,  note that the function $n\mapsto h(n)=\ell n- k \omega(n)$ is one-to-one on $\Z$ since $k$ and $\ell$ are relatively prime. Thus,  $\widehat F(m) =\widehat g(n) e^{i\omega(n)c}$ for $m=h(n) $, and zero otherwise. This immediately implies the claim since for large $n$,  $|m|=|h(n)|\sim |n|^d$ and $g\not \in H^{\frac12+}$.

 For $k= 1$, let $S=h(\Z)$. Using \eqref{eq:Fcoef} we have
\begin{multline*}
\|F\|_{H^r (\T)}^2=\sum_{m\in S} \la m\ra^{2r} \big|\sum_{n\in h^{-1}(m) }   \widehat g(n) e^{ic \ell n}\big|^2\\= \sum_{m\in S} \la m\ra^{2r} \sum_{n\in h^{-1}(m)} | \widehat g(n)|^2
+ 2\Re \sum_{m\in S} \la m\ra^{2r} \sum_{n\neq j \in h^{-1}(m)} \widehat g(n) \overline{\widehat g(j)} e^{ic \ell (n-j)}.
\end{multline*}
Note that the first sum is infinite for $r>\frac1{2d}$ since $g\not\in H^{\frac12+}$ and  $|m|\sim |n|^d$ for large $n$. Therefore, it suffices to see that the second summand is finite for a.e.~$c$, which follows if it belongs to $L^2(\T)$ as a function of $c$. We rewrite\footnote{The reordering of the sum can be justified by passing to a subsequence as in the proof of Lemma 2.18 in \cite{ETbook}.} the second sum as
$$
\sum_{a\neq 0} e^{ic\ell a} \sum_{j:h(j)=h(j+a)}\la h(j)\ra^{2r}\widehat g(j+a) \overline{\widehat g(j)}.
$$
Note that for each $a$ the inner sum is over at most $d-1$ values of $j$ since $h(j)-h(j+a)$ is a polynomial of degree $d-1$ in $j$.
Therefore, by Plancherel, the $L^2$  norm can be  bounded by the square root of 
\begin{multline*}
\sum_{a\neq 0}  \sum_{j:h(j)=h(j+a)}\la h(j)\ra^{4r}|\widehat g(j+a)|^2 |\widehat g(j)|^2\\ \les \sum_{a\neq 0}  \sum_{j:h(j)=h(j+a)} \la h(j)\ra^{4r-2/d} \la h(j+a)\ra^{2/d}  \la j+a\ra^{-2} \la j\ra^{-2} \\
\les \sum_{a\neq 0}  \sum_{j:h(j)=h(j+a)} \la j\ra^{4rd-4}\les \sum_j\la j\ra^{4rd-4},
\end{multline*}
which is finite provided that $r<\frac3{4d}$.
 This suffices since the $H^r$ spaces are nested. In the first inequality we used that $g\in BV(\T)$, which implies that $|\widehat g (n)|\les \frac1{\la n\ra}$. 
\end{proof}

The upper bound $2-\frac1{2d}$ would also follow if one can upgrade the $L^4$ upper bound in the proof above to the  $L^\infty$ norm instead, which is currently out of reach. The following proposition and its corollary provide a weaker upper bound.

\begin{prop} \label{prop:diag_infty}   Fix  $k ,  \ell\in \N$   with $(k,\ell) = 1$. Let $\omega$ be a polynomial of degree $d$ with integer coefficients. 
Then for a.e.~$c \in \R$ we have\footnote{One can improve this bound for large $d$ using  \eqref{vinmvt} instead of Weyl's bound in the proof.}  
$$
\Big\|\sum_{ N \leq n < 2 N }   e^{i(c-\frac{k}\ell x) \omega (n) + in  x} \Big\|_{L^\infty_x}\les  N^{1 - \frac{1}{2^d+1 } +  },\,\,\,\,\, N\in\N.
$$

\end{prop}
\begin{proof}
Let $t=c-\frac{kx}{\ell}$.
We fix $0 < \delta < 1$ to be chosen later. We aim to show that for a.e.~$c$
\begin{equation}\label{showthis} \Big|\sum_{ N \leq n < 2 N } e^{it\omega(n)+ixn} \Big|\lesssim N^{1-\frac{\delta}{2^{d-1}}+  },
\end{equation} 
By Dirichlet's box principle there is a natural number $q \leq N^{d-\delta}$ such that $t= 2\pi  \frac{a}{q} + \beta_t$ where $|\beta_t| \leq \frac{1}{qN^{d-\delta}}$ and $(a,q)=1$. By Weyl's bound (for instance Lemma 2.4 in \cite{Vau}), 
$$ \Big|\sum_{ N \leq n < 2N} e^{it\omega(n)+ixn} \Big| \lesssim  N^{1 + }(q^{-1} + N^{-1} + qN^{-d})^{2^{1-d}}.$$ 
Therefore, the bound \eqref{showthis} holds provided that $q\geq N^{\delta }$, and thus we may assume $q < N^{\delta }$. 

Weyl's bound only takes into account the leading term, so now we use summation by parts,
\begin{equation}\label{sumbyparts} \sum_{y < n \leq x} a_n h(n) = A(x) h(x) - A(y) h(y) - \int_y^x A(t) h'(t) dt,\,\,\,\,\,\,A(t) := \sum_{n \leq t} a_n,
\end{equation}   
to get cancellation from the linear term. We have that 
\begin{align*} \Big|\sum_{ N \leq n  < 2N} e^{it\omega(n)+ixn} \Big| &\lesssim (1 +| \beta_t| N^{d}) \sup_{ N \leq u < 2N} \Big| \sum_{ N \leq n \leq u}  e^{i nx + 2\pi i\frac{a\omega(n)}q}\Big| \\ 
& \lesssim  \frac{N^{\delta}}q \sup_{ N \leq u <2 N } \sum_{j=1}^q \Big|\sum_{\frac{ N-j}q \leq m \leq \frac{u-j}q}  e^{ i  mq x} \Big| \\
& \lesssim  N^\delta  \frac{1}{\|\frac{qx}{2\pi} \|}.
\end{align*}
In the first inequality we used summation by parts and
$$|\partial_n e^{i\omega(n)\beta_t} |\les |\beta_t|  |\omega^\prime (n)|  \les |\beta_t|   N^{d-1},\,\,\,\, N \leq n \leq 2 N.$$
In the second inequality we used the bound for $\beta_t$ and wrote $n=mq+j$, $j=1,2,...,q$.
Therefore, the bound \eqref{showthis} holds provided that $\|\frac{qx}{2\pi} \|\geq N^\delta 
N^{-1+\frac{\delta}{2^{d-1}} } $, and hence we may   assume that $\|\frac{qx}{2\pi} \| <  N^\delta N^{-1+\frac{\delta}{2^{d-1}} } $. 
Thus there is a $j \in \mathbb{Z}$ such that $x =  2\pi \frac{j}q + \beta_x$, where $|\beta_x| \leq \frac1q  N^\delta N^{-1+\frac{\delta}{2^{d-1}} } $. 

We now show that for a.e.~$c$ and for all sufficiently large $N$ depending on $c$, the above two cases are indeed the only possible cases. Let $c$ be chosen such that\footnote{A slight variation of  this argument can be used to prove that the statement holds for all Khinchin--L{\'e}vy $\frac{c}{2\pi}$.} the approximation $c = 2\pi \frac{r}s + O(\frac1{s^{2+\epsilon}})$ happens finitely often for each $\epsilon>0$. It is a classical result that this happens for a.e.~$c$  (see for instance Problem 1 in Section 1.7 of \cite{St}). We have that $c = 2\pi \frac{a+j}q + \beta_x + \beta_t.$
  Using the bounds for $\beta_x$, $\beta_t$ and that $q \leq N^{ \delta+ }$, we obtain 
\be\label{eq:betas}
|\beta_x + \beta_t| \leq \frac{1}{q}\big( N^\delta  N^{- 1+\frac{\delta}{2^{d-1}} } + \frac1{N^{d-\delta}}\big) \leq \frac{1}{q^{2+  }},
\ee
as long as $\delta < \frac{1 }{2 + 2^{1-d}}$. Since this inequality can hold only for finitely many $N$, the claim follows.
\end{proof}

\begin{rmk}\label{rmk:diag}
i) As a special case of the above, we have that for $r\in \Q$ and for a.e.~$c$, 
$$
\Big\|\sum_{ N \leq n < 2N}   e^{i(c-rx)n^2+ i n  x}\Big\|_{L^\infty_x}\les   N^{\frac{4}{5} +  },\,\,\,\,\, N\in\N.
$$
Improvements would be of interest. For $c= M\pi$, $M$ odd and $r\in\Z$ odd, Oskolkov \cite{O2} has the optimal bound  $N^{\frac12}$.

ii) It is evident from the proof of Proposition~\ref{prop:diag_infty} that the same bound is valid for shorter sums:
$$
\sup_{ N\leq m\les 2 N } \Big\|\sum_{ N \leq n \leq m}   e^{i(c-\frac{k}\ell x) \omega (n) + in  x}\Big\|_{L^\infty_x}\les  N^{1 - \frac{1}{2^d+1 } + },\,\,\,\,\, N\in\N.
$$
\end{rmk}
Using Proposition~\ref{prop:diag_infty} we obtain dimension upper bounds for the function $f_{g,c,k,l}(x)=e^{ itL_\omega} g\Big|_{ t=c-\frac{k}\ell x}$, see  \eqref{fgckl}.
\begin{cor}\label{cor:obliqueupper} Let $g$ be  a step function, and  fix  $k ,  \ell\in \N$   with $(k,\ell) = 1$. Let $\omega$ be a polynomial of degree $d$ with integer coefficients.    Then for a.e.~$c$, the $2\pi\ell$--periodic function $f_{g,c,k,l}$ is in $  C^{\alpha}$, for every $\alpha<\frac1{d(2^d+1)}$. In particular, the dimension of the graph of real and imaginary parts is at most $2-\frac1{d(2^d+1)}$.  
\end{cor}
\begin{proof}
Using Proposition~\ref{prop:diag_infty}, Remark~\ref{rmk:diag}, and summation by parts \eqref{sumbyparts},
we conclude that for each $y\in \T$, and for a.e.~$c$ and for large $N$, we have 
(with $h(n)=\ell n - k\omega(n)$)
\be\label{eq:cyinf}
\sup_{x} \Big| \sum_{n:  N\leq |h(n)| <2 N }\frac1n e^{in(\ell x-y)- i k\omega(n)  x+ i\omega(n)c}\Big|   \les \frac1{N^{\frac1d}}  N^{\frac{1}d(1 - \frac{1}{2^d+1 } + )}= N^{ -\frac1{d(2^d+1)}+  }.
\ee
Here the set of $c$ for which the bound is valid depends on $y$. Now note that
\begin{multline*}
F(x)=f_{g,c,k,\ell}(\ell x)=\sum_n \widehat{g}(n) e^{in\ell x+i\omega(n)t}\Big|_{t=c- k x}  \\
=  \sum_n \frac1n \widehat{dg}(n) e^{in\ell x+i\omega(n)t}\Big|_{t=c-k x} =
\int_\T H(c- k  x,\ell x-y) dg(y),
\end{multline*}
where $H(t,x)= \sum_n \frac1n e^{inx+i\omega(n)t}$.  The  bound \eqref{eq:cyinf} implies that for each $y$ and a.e.~$c$, 
$H(c- k  x,\ell x-y)$ is a $C^\alpha$ function of $x$ for $\alpha<\frac1{d(2^d+1)}$. Therefore, if $g$ is a step function, then for a.e.~$c$, $F$ is a finite sum of $C^{\alpha}$ functions. Thus,  $F$ is $C^{\alpha}$ for $\alpha<\frac1{d(2^d+1)}$.
\end{proof}

 \begin{rmk}\label{rmk:vert}
One can easily adopt the techniques above to understand the restriction of the solutions in the case of polynomial dispersion  to vertical lines. In particular, we have 
 for a.e.~$x$
$$
\Big\|\sum_{ N \leq n < 2N}   e^{it \omega (n) + in  x}\Big\|_{L^\infty_t}\les  N^{1 - \frac{1}{2^{d }+1} + },\,\,\,\,\, N\in\N.
$$
The proof is identical to the proof of Proposition \ref{prop:diag_infty}, the only change is we use \eqref{eq:betas} only for $\beta_x$. This implies the upper bound $2-\frac1{d(2^d+1)}$ for the dimension. 

Similarly, the lower bound we have in   Proposition~\ref{prop:obliquelower} is valid for a.e.~$x$. Thus, for step function initial data the dimension lies in the interval $[2-\frac1{2d},  2 - \frac1{d(2^d+1)}]$ for a.e.~$x$. 

In the case of Schr\"odinger evolution, the upper bound above is weaker than the one given in \cite{OC13}, 
which together with our lower bound imply that for step function data, the dimension on almost every vertical 
line is in $[\frac74,\frac{15}8]$.  
 \end{rmk}
 
As it was stated in Theorem~\ref{thm:NLSoblique},   one can extend these results to certain nonlinear  variants:
\begin{proof}[Proof of Theorem~\ref{thm:NLSoblique}]
Consider the  Wick ordered  cubic NLS equation \eqref{NLSwick}.
 The smoothing estimate in \cite{egtz1} implies that
for $g\in BV(\T)\subset  H^{\frac12-}(\T)$,
\be\label{c0c1-}
N(t,x):=u(t,x)-e^{it\partial_{xx} }g\in C^0_tH^{\frac32-}_x\subset C^0_tC^{1-}_x .
\ee
 Using the Duhamel's formula this implies that $N(t,x) \in C^1_tH^{-\frac12-}_x$, in particular
$$
 \|N(t_1,\cdot)-N(t_2,\cdot)\|_{H^{-\frac12-}}\les |t_1-t_2|.
$$
Interpolating this with \eqref{c0c1-} and using Sobolev embedding, we conclude that  
$$
\sup_x  |N(t_1,x)-N(t_2,x)|\les  \|N(t_1,\cdot)-N(t_2,\cdot)\|_{H^{ \frac12+}} \les |t_1-t_2|^{\frac12-}.
$$
This together with  \eqref{c0c1-} imply that    $N(t,x)\in C^{\frac12-}_{t,x}$, locally\footnote{For each $\alpha<\frac12$, the $C^{\alpha}_{t,x}$ norm   is finite on each compact subset of $\R\times \T$ although it may grow in--time}  in $t$, and hence
by Corollary~\ref{cor:obliqueupper}, for a.e.~$c$  the function $F_c(x)=u(c-\frac{k}\ell x,x)$ is the sum of a  $C^{\frac12-}$ function and a $C^{\frac1{10}-}$  function, which implies the upper bound.

 For the lower bound one needs to be slightly careful since the nonlinear part of the evolution is not periodic. This can easily be remedied by taking an even $4\pi \ell$--periodic extension of the nonlinear part $N(c-\frac{k}\ell x,x)$ which also belongs to $C^{\frac12-}= B^{\frac12-}_{\infty,\infty}\subset B^{\frac12-}_{1,\infty}$. Since, by the proof of Proposition~\ref{prop:obliquelower}, the linear part is not in 
$B^{\frac14+}_{1,\infty}$, the sum does not belong to $B^{\frac14+}_{1,\infty}$.
This yields the lower bound for the dimension of the graph.   
\end{proof}
 
The proof of Theorem~\ref{thm:kdvoblique} is similar.    
The smoothing estimate in \cite{ET1} implies that 
$$
N(t,x):=u(t,x)-e^{-t\partial_{xxx} } g \in C^0_tH^{\frac32-}_x.
$$ 
The argument above leads to $N(t,x)\in C^{\frac13-}_{t,x}$, which suffices to extend the  bounds for the Airy equation (the case $d=3$ in Proposition~\ref{prop:obliquelower}, Corollary~\ref{cor:obliqueupper}, and  Remark~\ref{rmk:vert}) to the KdV evolution \eqref{KdV}.

\subsection{Results on space slices} \label{sec:poly_space}

Suppose $\omega (n)$ is a polynomial of degree $d$  with integer coefficients. First, using the recent progress on Vinogradov's mean value theorem, we improve upon some bounds on $D_t(\omega,g)$ obtained in   \cite{cet} (also see \cite{ETbook}) for large $d$. In addition, for the Airy equation, $\omega (n) = n^3$, we are able to utilize a factorization to improve upon the lower bound for the fractal dimension proving Theorem~\ref{thm:kdv3/2}.
 
\begin{theorem} Let  $\omega (n)$ be a polynomial of degree $d$  with integer coefficients and leading coefficient $a_d$. Then for any $g\in BV(\T)\setminus H^{\frac12+}(\T)$, we have 
$D_t(\omega,g)\in [1+\frac1{d(d-1)},2-\frac1{d(d-1)}]$ provided that $\frac{ta_d}{2\pi}$ is Khinchin--L\'evy.
\end{theorem}
In \cite{cet}, the author obtained the bound $D_t(\omega,g)\in [1+2^{1-d},2-2^{1-d}]$.
\begin{proof}
By Theorem~\ref{DJcor}, it suffices to prove that   
$$\|H_{N,\omega}^\pm(x,t)\|_{L^\infty_x} \les_{\epsilon} N^{1-\frac1{d(d-1)} +}$$
for  Khinchin--L\'evy  $\frac{ta_d}{2\pi}$.

 Let $\gamma_{d-1}(x) := (x^{d-1}, \ldots , x)$ be the twisted cubic. Let $$J_{s , d-1}(2^N) : = \big\|\sum_{1 \leq n \leq 2^N} e^{i\gamma_{d-1}(n) \cdot \vec{x} }\big\|_{L^{2s}(\T^{d-1})}^{2s}.$$ Bourgain, Demeter and Guth \cite{bdg} as well as Wooley \cite{Woo} used decoupling and efficient congruencing, respectively, to show the optimal bound (up to the $N^{\epsilon})$ $$J_{{d \choose 2} , d-1}( N) \lesssim  N ^{{d \choose 2}+}.$$ It turns out that this mean value estimate has applications to individual exponential sums. We let $\frac{t\omega (n)}{2\pi} = \alpha_d n^d + \ldots + \alpha_1 n$. We may assume without loss of generality that $\omega (0) = 0$.

\begin{theorem}[Theorem 5.2 in \cite{Vau}] Fix $s\in \N$. Suppose that there exist $j,a,q$ with $2 \leq j \leq k$, $|\alpha_j - \frac{a}q| \leq \frac1{q^2}$, $(a,q) = 1$, $q \leq  N^j $. Then $$\big|\sum_{1 \leq n \leq  N}e^{it \omega (n) + inx}\big|^{2s} \lesssim J_{s,d-1}(2 N )  N^{d \choose 2} (q N^{-j} +  N^{-1} + q^{-1})\log ( 2N ).$$
\end{theorem}

Applying this with $s = {d \choose 2}$, $j=d$ and using Vinogradov's mean value theorem to estimate $J_{{d \choose 2} , d-1}(2 N )$, we find that \begin{equation}\label{vinmvt}\big|\sum_{1 \leq n \leq  N}e^{it \omega (n) + inx }\big| \lesssim_{\epsilon , d} N^{1 +  }(q N^{-d} + N^{-1} + q^{-1})^{\frac1{d(d-1)}}.\end{equation}

  Since $\alpha_d =\frac{ta_d}{2\pi}$ is Khinchin--L\'{e}vy. Then, given $\epsilon>0$,  we may choose $q \in [N , N^{1+\epsilon}]$ and the theorem follows for such $t$, as $$\|H_{N,\omega}^+(x,t)\|_{L^\infty_x} =\Big\|\sum_{1 \leq n <  2N } e^{it \omega (n) + inx }  - \sum_{1 \leq n \leq  N} e^{it \omega (n) + inx}\Big\|_{L^\infty_x}.$$
\end{proof}

For the Airy equation, we are able to take advantage of a special factorization for cubic polynomials to improve upon the lower bound presented in \cite{cet}. This can also be obtained from Strichartz estimates for Airy  equation  via Theorem \ref{thm:strichloss} below. We provided Theorem~\ref{thm:kdv3/2}   as an alternative method, which also gives more exact information on the set of $t$ for which the fractal dimension is at least $\frac32$ (Khinchin-L\'{e}vy $\frac{t}{2\pi}$). 
\begin{proof}[Proof of Theorem~\ref{thm:kdv3/2}] 
 The claim for the KdV evolution follows from the statement for Airy and the smoothing estimate in \cite{ET1} as in Section~\ref{sec:oblique}. For the Airy evolution, by Theorem \ref{DJcor}, it is enough to show for   Khinchin-L\'{e}vy $\frac{t}{2\pi}$ that
\be\label{L4airy}
\Big\|\sum_{ N\leq n \leq 2N}   e^{inx+in^3t} \Big\|_{L_x^4}\les  N^{\frac12+},
\ee
and a similar bound for $n\in [-2 N ,- N]$.
 
Now note that  
\begin{multline*}
\Big\|\sum_{ N\leq n \leq 2N} e^{inx+in^3t} \Big\|_{L^4_x}^4=2\pi\sum_{n_1,n_2,n_3= N}^{2 N }e^{i(n_1^3-n_2^3+n_3^3-(n_1-n_2+n_3)^3)t}\\
=2\pi \sum_{ n_1,n_2,n_3= N }^{2 N } e^{3it(n_1-n_2)(n_2-n_3)(n_1+n_3)}.
\end{multline*}
Letting $j=n_1-n_2$ and $k=n_2-n_3$, we rewrite this as 
$$
\sum_{n_2= N}^{2 N }\sum_{j= N-n_2}^{2 N -n_2}\sum_{k=n_2-2 N }^{n_2- N}  e\big(3tjk(j-k+2n_2)\big).
$$
Changing the order of the sums we rewrite
\begin{multline*} 
\sum_{j=- N}^{2 N } \sum_{n_2 = \max( N-j, N)}^{\min(2 N -j,2 N )} \sum_{k=n_2-2 N }^{n_2- N}  e^{3itjk(j-k+2n_2)} \\ =\sum_{j=- N}^{ N}  \sum_{k=\max(- N-j,- N)}^{\min( N-j, N)} 
\sum_{n_2=\max( N-j, N,k+ N)}^{\min(2 N -j,2 N ,k+2 N )} e^{3itjk(j-k+2n_2)}.
\end{multline*}
Thus we can estimate the sum by
$$
\les \sum_{j=- N}^{ N}  \sum_{k=\max(- N-j,- N)}^{\min( N-j, N)} \min\big( N,\frac1{\|\frac{3 jkt}{\pi}\|}\big) \\ \les N^{0+}\sum_{|w|\les  2N }\min\big( N,\frac1{\|\frac{wt}{2\pi}\|}\big)\les N^{2+ } 
$$
for any Khinchin--L{\'e}vy $\frac{t}{2\pi}$ finishing the proof. In the last inequality we used Lemma 2.2 in \cite{Vau}.
\end{proof}

The lower bounds for the dimension can also be obtained using Strichartz estimates. Note that the statement below does not require that the data is of bounded variation. 
\begin{theorem}\label{thm:strichloss}
Assume that $e^{itL_\omega}$ satisfies a Strichartz estimate of the form
\be\label{strich}
\|e^{itL_\omega} f\|_{L^p_tL^q_x}\les \|f\|_{H^s}
\ee
for some $s\in [0, \frac12)$, $2<q<\infty$ and $p\in [1,\infty]$. Let $r_0:=\sup\{r: g\in H^r\} >\frac12$. 
Then for a.e.~$t$,  $D_t(\omega,g)\geq  2-\frac{2r_0-(r_0-s)q^\prime}{2-q^\prime}$. In particular for $s=0$, $D_t(\omega,g)\geq  2- r_0 $. 
\end{theorem}
\begin{proof}
First of all note that by Sobolev embedding and the continuity of the propagator, $e^{itL_\omega}g$ is a continuous function of $x$ and $t$. 
By the Strichartz estimate, we  have
$$
\|\la \partial_x\ra^{r_0-s-} e^{itL_\omega} g\|_{L^p_tL^q_x}= \| e^{itL_\omega} \la  \partial_x\ra^{r_0-s-}g\|_{L^p_tL^q_x}\les \|\la \partial_x\ra^{r_0-}g\|_{L^2}<\infty.
$$
Thus for a.e.~$t$,  
$\|\la \partial_x\ra^{r_0-s-} e^{itL_\omega} g\|_{L^q_x}< \infty$. By 
Littlewood--Paley theory, this implies that for a.e.~$t$ 
$$
\|P_Ne^{itL_\omega}g\|_{L^q_x} \les_t N^{-r_0+s+}
$$
Now, the statement follows from the interpolation argument as in the proof of Theorem~\ref{DJcor} above.   
\end{proof}
\begin{rmk} Recall that the linear Schr\"odinger and Airy propagators   satisfy the Strichartz estimates above with $p=q=4$ and $s=0$, see e.g. \cite{B1} or \cite{ETbook}.  Therefore, the lower bound $2-r_0$ follows.
In the case $r_0=\frac12$, one needs to assume in addition  that $g\in BV(\T)$, which by a theorem of Oskolkov \cite[Proposition 14]{O} implies that $e^{itL}g$ is continuous in $x$ for a.e.~$t$ for both Schr\"odinger  and Airy propagators.    

\end{rmk}

\section{Equations with non--polynomial dispersion relations}\label{sec:nonpoly}
We start this section by noting that the results for the polynomial dispersion relation apply to some equations with non--polynomial dispersion relation if the dispersion relation is a small perturbation of a polynomial. For example for the Boussinesq equation the dispersion relation is 
$\omega(n)= \sqrt{n^2+n^4} =n^2+ r(n) $, where 
$r(n)=O(1)$ and $r^\prime(n) =O(\frac1n)$. Therefore by summation by parts \eqref{sumbyparts}  we have 
$$
\Big|\sum_{ N\leq n< 2 N }e^{i\omega(n)t+inx} \Big| \les \sup_{ N\leq u< 2 N } \Big| \sum_{ N\leq n \leq u} e^{in^2t+inx}\Big|
$$
which, for a.e.~$t$, is bounded by $N^{ \frac12+ }$ uniformly in $N$. In particular, as for the 
Schr\"odinger equation, for a.e.~$t$  the solution to the Boussinesq equation on the torus has dimension $\frac32$ for step function initial data or more generally for data in $BV(\T)\setminus H^{\frac12+}(\T)$. 
The same statement holds for the linear group of Benjamin--Ono equation whose dispersion relation is $n|n|$.  

Below we study equations with $\omega(n)=|n|^\alpha, $ $\alpha>0$ noninteger. We note applications to water wave equations: For the gravity--capillary wave equation, we have 
$$\omega(n)=\sqrt{(n+n^3)\tanh(n)}=|n|^{\frac32}+r(n),$$ where $|r(n)|\les 1$ and $|r^\prime(n)|\les |n|^{-\frac32}$. Therefore, by summation by parts,  the bounds we prove below, see Theorem~\ref{thm:nalpha} and Theorem~\ref{t3/2}, for $\omega(n)=|n|^{\frac32}$  apply to gravity--capillary wave equation. In particular, for step function data $g$ we have $D_t(\omega,g)\in [\frac54,\frac74]$ for all $t\neq 0$ and $D_t(\omega,g)\in[\frac{11}8,\frac{13}8]$ for a.e.~$t$ characterized by Khinchin--L\'evy numbers.

Similarly, for the gravity water waves $\omega(n)=\sqrt{n \tanh(n)},$  the assertion (for $\alpha=\frac12$) of Theorem~\ref{thm:nalpha}  is valid.  In particular,  for step function data $g$,  $D_t(\omega,g)\in[\frac54,\frac74]$ for all $t\neq 0$.

\subsection{Fractional  Schr\"odinger equation}
In this section, we  consider the fractional cubic Schr\"odinger equation:
\begin{equation}\label{fsch}
\left\{
\begin{array}{l}
iu_{t}+(-\Delta)^{\frac\alpha2} u =\pm |u|^2u, \,\,\,\,  x \in \T, \,\,\,\,  t\in \mathbb{R} ,\\
u(x,0)=g(x)\in H^{s}(\T), \\
\end{array}
\right.
\end{equation}
where $\alpha \in (0,2)\setminus\{1\}$. The case $\alpha =2$ is the cubic NLS equation.  

 In \cite{det,egtz1} smoothing estimates for this equation were established. In particular,  for $\alpha\in (1,2)$,
Theorem 1.1 in \cite{egtz1} implies that if $g\in BV(\T)\subset   H^{\frac12-}(\T),$  then  
 $$
 u(t,x)-e^{it(-\Delta)^{\frac\alpha2}-iPt} g \in C^0_tH_x^{s_1}
 $$
 for $t$ in the local existence interval and for $s_1<\alpha-\frac12$.
 Therefore,  as a function of $x$,  $u(t,x)-e^{it(-\Delta)^{\frac\alpha2}-iPt} g$ is in $C^\beta$ for each $\beta<\alpha-1$.  
 
With this observation we concentrate on the properties of the linear group
$$
e^{it(-\Delta)^{\frac\alpha2} } g = \sum_n \widehat g(n) e^{i t |n|^{\alpha} +i xn }.
$$
We start with the proof of Theorem~\ref{thm:nalpha} which is a   simple corollary of van der Corput bounds and   Theorem~\ref{DJcor} above. In sharp contrast with the case of polynomial dispersion the assertion of this theorem  is independent of the algebraic nature of time. This holds since $\{t|n|^\alpha (\text{mod }1): n\in \Z\}$ is uniformly distributed on $[0,1]$ for $\alpha\in (1,2)$ and  for  $t\neq 0$.
In Section~\ref{sec:ep} we  will discuss possible improvements of this result that hold for almost every time, or for certain values of $\alpha\in (1,2)$. 
\begin{proof}[Proof of Theorem~\ref{thm:nalpha}]
Fix  $\alpha\in(0,2)\setminus \{1\}$.   
We claim that for each $t\neq 0$  
\begin{equation}\label{expsum1}\Big\|   \sum_{ N \leq n < 2 N }e^{i(t n^{\alpha} \pm xn)} \Big\|_{L_x^{\infty}} \lesssim N^{1-\beta}.\end{equation} 
 This yields the proposition by Theorem~\ref{DJcor}.

We use the following van der Corput bound, see for example  Theorem 8.20 in \cite{Iw}:
\begin{theorem} For given real valued $f$ and fixed integer $k\geq 2$, 
$$
\big|\sum_{n \sim  N} e^{if(n)}\big|\les N\big( \Lambda^{\frac1{2^k-2}}    + \Lambda^{-\frac1{2^k-2}} N^{ -2^{2-k}}\big)
$$
provided that $f\in C^k$ and $ |f^{(k)}(u)| \sim \Lambda>0$ for $u\sim  N$. 
\end{theorem}  
We use this with $f(u)=tu^\alpha\pm ux$ and $\Lambda \sim  |t| N^{\alpha-k }$  taking $k=2$ for $\alpha \leq \frac32$ and $k=3$ for $\alpha>\frac32$. 
\end{proof}
 
Combining this result with the smoothing estimate from \cite{egtz1} discussed above we have 
\begin{theorem}\label{thm:fracNLS} Fix $\alpha\in (1,2)$.  The claim of Theorem~\ref{thm:nalpha} remains valid for the cubic fractional  Schr\"odinger evolution \eqref{fsch} on the local existence interval for any 
$$\beta < \left\{\begin{array}{ll}
\alpha-1,&  \alpha \in(1,\frac43),\\
 1-  \frac\alpha2, & \alpha\in [\frac43,\frac32],\\
  \frac12 - \frac\alpha6, &  \alpha\in( \frac32,2).
  \end{array}\right.
$$ 
\end{theorem}
  
 \begin{rmk}  In \cite{det}, it was proved that the fractional Schr\"odinger evolution $e^{it(-\Delta)^{\frac\alpha2}}$, $\alpha\in(1,2)$, satisfies the Strichartz estimates \eqref{strich} with $p=q=4$ and $s>\frac{2-\alpha}8$. Therefore, by Theorem~\ref{thm:strichloss},  for a.e.~$t$, $D_t(|n|^\alpha,g)$ is at least $\frac32-r_0+\frac\alpha4$ 
 when  $r_0:=\sup\{r: g\in H^r\} >\frac12$. 
In the case $r_0=\frac12$, one needs to assume in addition  that $g\in BV$, which by Theorem~\ref{thm:nalpha} implies that $e^{it(-\Delta)^{\frac\alpha2}}g$ is continuous in $x$ for every $t$. Thus, in particular for step function data the dimension is  at least $1+\frac\alpha4$ for a.e.~$t$. This improves the lower bound for the dimension that Theorem~\ref{thm:nalpha} and Theorem~\ref{thm:fracNLS} give for all $\alpha\in(\frac43,2)$ and  for a.e. $t$. 
\end{rmk}

\subsection{Higher order non--polynomial dispersion relations} 

In this section we consider dispersion relations $\omega(n)=|n|^\alpha,  \alpha\in (2,\infty)\setminus \N$. The following theorem of Heath--Brown \cite{He}, which improves the van der Corput bound above for $k\geq 3$,  is useful in this range:
\begin{theorem}[\cite{He}]\label{higherdeg}
For given real valued $f$ and fixed integer $k\geq 3$, 
$$
\big|\sum_{n \sim  N} e^{if(n)}\big|\les N^{1+}   \big(\Lambda +  N^{-1 } 
+ N^{-2  }\Lambda^{-\frac2{k }}\big)^\frac1{k(k-1)}
$$
provided that $f\in C^k$ and $ |f^{(k)}(u)| \sim \Lambda>0$ for $u\sim  N$.  
\end{theorem}
For $d\geq 3$,  $\alpha \in (d-1,d)$ and $\{\alpha\}>\frac2{d+1}$ we apply the theorem with $k=d+1$ to get 
 $$\big\|H_{N,\omega}^\pm\big\|_{L^\infty_x}\les N^{1-\frac1{d(d+1)}+},\,\,\,t\neq 0.$$
 For $\{\alpha\}\leq \frac2{d+1}$, we apply it with $k=d$ to get $$ \big\|H_{N,\omega}^\pm\big\|_{L^\infty_x}\les N^{1-\frac{1-\{\alpha\}}{d(d-1)}+ },\,\,\,t\neq 0.$$  
This leads to dimension bounds as above. In particular, it implies that for $g\in BV(\T)$, $e^{it(-\Delta)^{\frac\alpha2}}g$ is continuous in $x$ for each $t\neq 0$, and hence quantization fails for each noninteger $\alpha>0$.

 \subsection{Exponent pair conjecture, and $A$ and $B$ processes}\label{sec:ep}
  
We remark that for irrational $\alpha > 0$, the theory of exponent pairs (see chapter 3 of \cite{Gr} or section 8.4 in \cite{Iw}) applies. An exponent pair $(k,\ell)$ as defined in chapter 3 in \cite{Gr}, would give rise to the bound 
$$\|H_{N,\omega}^{\pm} \|_{L_x^{\infty}} = \Big\| \sum_{N\leq n <2  N} e^{itn^{\alpha} \pm ixn} \Big\|_{ L_x^{\infty}}\lesssim  N^{k\alpha + \ell - k }.$$  
The exponent pair conjecture claims that $(0,1/2 + \epsilon)$ is an exponent pair, which would imply square root cancellation.
 This is a problem central to number theory, as it would imply the Lindel\"of hypothesis, the Gauss circle problem, and a host of other famous problems in number theory. Although the conjecture is far from being settled, there are many partial results towards it.

There are two methods, called the $A$ and $B$ processes, which combine to give rise to an infinite family of exponent pairs, see e.g. Chapters 3 and 5 of \cite{Gr}. For example $(\frac19,\frac{13}{18})$ is an exponent pair, see Chapter 7 of \cite{Gr}. In order to understand the Riemann zeta function on the critical line, specialized methods have been developed (see Chapter 7 of \cite{Gr}), culminating in a recent paper of Bourgain \cite{Bo} which yields the  exponent pair $(\frac{13}{84}+,\frac{55}{84}+)$. These bounds are the best known for irrational $\alpha$ around 2.

 Finally, we note that for rational values of $\alpha$ one cannot blindly apply the results on the exponent pair conjecture. However,   one may apply the $A$ and $B$ processes directly to improve the bounds discussed above.   We give an example of such an argument below  when $r= \frac\alpha{\alpha-1}$ is an integer. Unlike the previous result (Theorem~\ref{thm:nalpha}), the algebraic nature of time will be important in what follows.  

In number theory,  it is standard to study 1-periodic functions as opposed to $2\pi$-periodic functions that we consider here. To remain consistent with the bounds that we are using from number theory texts that we cite, we scale $t$ and $x$ by $2\pi$. Recall that for $\theta \in \mathbb{R}$, $e(\theta) := e^{2\pi i \theta}$. Our goal is to estimate the $L^\infty$ norm of 
$$
H_N^+ =\sum_{ N \leq n < 2 N }e(t n^{\alpha} + xn).
$$

\begin{theorem}\label{t3/2} Fix $\alpha\in (1,2) $ so that  $r= \frac\alpha{\alpha-1}$ is an integer.
Let $t \in \mathbb{R}$ such that $c_{t,r}:= t^{1-r} (r-1)^{r-1} r^{-r} $ is Khinchin--L{\'e}vy. Then $$\sup_{x \in \T} | \sum_{n \sim  N} e(tn^{\alpha} + x n) | \lesssim_{\alpha , t}  N^{ \frac{ \alpha}{2}+ }  N^{- (\alpha-1) 2^{1-r}}.$$ In particular the bound holds for a.e.~$t$. 
\end{theorem}
\begin{proof} We first  use the B-process of Van der Corput, that is Theorem 8.16 in \cite{Iw} which we state for convenience. 
\begin{theorem}[Theorem 8.16 in \cite{Iw}] Let $f(x)$ be a real function on $[a,b]$ whose derivatives satisfy the following conditions: $\Lambda \leq f'' \leq \eta \Lambda$, $|f^{(3)}| \leq \eta \Lambda (b-a)^{-1}$, $|f^{(4)}| \leq \eta \Lambda (b-a)^{-2}$ for some $\Lambda > 0$ and $\eta \geq 1$. Then we have $$\sum_{a < n < b} e(f(n)) = \sum_{f'(a) < m < f'(b)} e(f(x_m) - mx_m + 1/8) f''(x_m)^{-1/2} + R_f(a,b), $$ where $x_m$ is the unique solution to $f'(x) = m$. The  error term  satisfies 
$$R_f(a,b) \lesssim \Lambda^{-1/2} + \eta^2 \log (f'(b) - f'(a) + 1),$$ 
where the implied constant is absolute.
\end{theorem}

Thus the original exponential sum is asymptotically equal to its ``dual" exponential sum, which arises from an application of Poisson summation formula and stationary phase estimates.  In our case, $f(u) = tu^{\alpha} + x u$, $a=   N  -1$, $b = 2 N $, $x_m = \left( \frac{m-x}{t \alpha} \right)^{\frac1{\alpha-1}}$, and using $\alpha = \frac{r}{ r-1}$ this simplifies to
\begin{align}\label{maineq1}\begin{split}  H_N^+(t,x)  
&= C \sum_{f'( N) \leq m \leq f'(2 N )} (m-x)^{\frac{2 - \alpha}{2\alpha - 2}}    e\left(t x_m^{\alpha} + (x-m) x_m \right) + O( N^{ 1 - \frac\alpha2 }) 
\\
&=  C \sum_{f'( N) \leq m \leq f'(2 N )} (m-x)^{\frac{r}2-1}     e\left(c_{t,r} (m-x)^{r}  \right) + O(N^{\frac12}),\end{split} \end{align}
where  $c_{t,r}= t^{1-r} (r-1)^{r-1} r^{-r} $ and $C=C(\alpha,t) $ is a constant. Since the error term is $O(N^{\frac12})$, we concentrate on the main term above. Our assumption on $\alpha$  ensures that the amplitude function is a polynomial. 

Using summation by parts \eqref{sumbyparts} 
and noting that $f'( N) \sim  f'(2 N ) \sim  N^{\alpha-1 }=N^{\frac{1}{r-1}}$, we bound the main term by 
\be\label{ctrsum}
  N^{ (\alpha-1)(\frac{r}2-1) } \sup_{f'( N) \leq u \leq f'(2 N )}  \Big| \sum_{f^\prime( N) \leq m \leq u} e\left(c_{t,r} (m-x)^{r}  \right) \Big|.
\ee
Now we are in a position to apply\footnote{One could also apply \eqref{vinmvt} to obtain better bounds for $r \geq 7$, but this inequality is worse for $r=3$ which we consider the most interesting case.} Weyl's inequality (for instance, Lemma 2.4 in \cite{Vau}) to   bound  the quantity above by
$$
  N^{(\alpha-1)(\frac{r}2-1) } N^{  \alpha-1+ } \big(q^{-1}+ N^{-  \alpha+1 }+ q  N^{-r (\alpha-1) }  \big)^{2^{1-r}},
$$
where $q \in \mathbb{N}$ is defined by the rational approximation $|c_{t,r} - \frac{a}q| \leq \frac1{q^2}$, with $(a,q)=1$. By assumption, $c_{t,r}$ is Khinchin--L{\'e}vy, and we may choose $q \in [N^{\frac{  \alpha - 1 }{2}} ,  N^{\frac{  \alpha - 1 }{2}+}]$ to obtain
$$
|H_N^+(t,x)| \lesssim N^{ \frac{ \alpha}{2}+ } N^{- (\alpha-1) 2^{1-r}}.
$$
This finishes the proof of Theorem~\ref{t3/2}.
\end{proof}

\begin{rmk}\label{rmk:32}
For the sake of discussion, we specialize further to the case $\alpha = \frac32$ (so $r = 3$). Then the above bound is $N^{\frac58}$. Note that this improves upon the bound \eqref{expsum1}  by a factor of $N^{\frac18}$ and gives the dimension estimate $\frac{11}8\leq D_t(\omega,g)\leq \frac{13}8$ improving the range  given in Theorem~\ref{thm:nalpha} for a.e.~$t$ characterized by Khinchin--L\'evy numbers.  If one was able to prove square root cancellation in cubic Weyl sums, the argument above would have implied square root cancellation in Theorem \ref{t3/2} leading to $D_t(\omega,g)=\frac32$.

We also remark that for $c_{t,3} = \frac{a }q \in \mathbb{Q}$, the supremum in \eqref{ctrsum} is at least (taking $x=0$)
$$\Big|\sum_{ m \sim  N^{\frac12 } } e \big( \frac{a m^3}{q}\big)\Big| \approx \frac{N^{\frac12}}{q}\Big| \sum_{j=1}^q e\big(\frac{ aj^3}{q} \big)\Big|.$$ 
It follows from the above considerations that the best bound we can obtain for the $L^\infty$ norm of $H_N$ is $N^{\frac34}$. This shows that for rational $t$, one cannot use the argument above   to improve upon the fractal dimension upper bound of $\frac74$. This is interesting, because just by looking at the original sum $\sum_{n \sim N} e(t n^{\frac32} + nx)$, there does not seem to be an obvious distinction between rational and Khinchin--L{\'e}vy $t$. It is not clear to us whether the fractal dimension should depend on the algebraic nature of $t$ or if this is simply an outcome of the techniques we used.
\end{rmk}


\begin{thebibliography}{99}
\bibitem[Be]{mber} M.~V.~Berry, {\it Quantum fractals in boxes,} J. Phys. A: Math. Gen. {\bf 29} (1996), 6617--6629.
\bibitem[BeKl]{berklei} M.~V.~Berry and S.~Klein, {\it Integer, fractional and fractal Talbot effects,} J. Mod. Optics { 43} (1996), 2139--2164.
\bibitem[BeLe]{berlew} M.~V.~Berry and Z.~V.~Lewis, {\it On the Weierstrass-Mandelbrot fractal function,} Proc. Roy. Soc. London A, {  370} (1980), 459--484.
\bibitem[BMS]{bermar} M.~V.~Berry, I.~Marzoli, and  W.~Schleich, {\it Quantum carpets, carpets of light,} Physics World {  14} (6) (2001), 39--44.
\bibitem[Bo1]{B1} J. ~Bourgain, {\it Fourier transform restriction phenomena for certain lattice subsets and applications to nonlinear evolution equations,} Geometric and Functional Analysis, Vol. 3 (1993), 107--156.
\bibitem[Bo2]{Bo} J.~Bourgain, {\it Decoupling, exponential sums and the Riemann zeta function,} J. Amer. Math. Soc. 30 (2017), no. 1, 205--224.
\bibitem[BDG]{bdg} J. Bourgain, C. Demeter, and L. Guth, {\it Proof of the main conjecture in Vinogradov's mean value theorem for degrees higher than three}, Ann. of Math. (2) 184 (2016), no. 2, 633--682.
\bibitem[ChCo]{ChCo} F. Chamizo A. Cordoba, {\it Differentiability and Dimension of Some Fractal Fourier Series} Adv. in Math. 142 (1999) 335--354.


\bibitem[ChOl1]{chenolv1}  G.~Chen and P.~J.~Olver, {\it Dispersion of discontinuous periodic waves,} Proc. Roy. Soc. London A 469 (2012), 20120407, 21pp.
\bibitem[ChOl2]{chenolv} G.~Chen and P.~J.~Olver, {\it Numerical simulation of nonlinear dispersive quantization,} Discrete Contin. Dyn. Syst.  {  34}  (2014),  no. 3, 991--1008.
\bibitem[CET]{cet} V.~Chousionis, M.~B.~Erdo\u{g}an, and N. Tzirakis, {\em Fractal solutions of linear and nonlinear dispersive partial differential equations},   Proc. Lond. Math. Soc.  (3)   { 110}  (2015), 543--564.
\bibitem[DeJa]{DJ} A.~Deliu  and B.~Jawerth, {\it Geometrical dimension versus smoothness,} Constr. Approx. { 8} (1992), 211--222. 
\bibitem[DET]{det}   S.~Demirba\c{s}, M.~B.~Erdo\u{g}an, and N.~Tzirakis, {\em Existence and uniqueness theory for the fractional Schr\"odinger equation on the torus}, Some topics in harmonic analysis and applications, 145–162,  Adv. Lect.  Math.  (ALM), 34, Int. Press, Somerville, MA, 2016. 

\bibitem[EGT]{egtz1} M.~B.~Erdo\u{g}an, B. G\"urel,  and N.~Tzirakis, {\it Smoothing for the fractional Schr\"odinger equation on the torus and the real line},   23pp, accepted by   Indiana Univ. Math. J. 
\bibitem[ErTz1]{ET1} M.~B.~Erdo\u gan and N.~Tzirakis, {\it Global smoothing for the periodic KdV evolution,} Int. Math. Res. Not. (2012), rns189, 26pp, doi: 10.1093/imrn/rns189.
\bibitem[ErTz2]{ET2}  M.~B.~Erdo\u gan and N.~Tzirakis,   {\it Talbot effect for the cubic nonlinear Schr\"odinger equation on the torus,}  Math. Res. Lett. {  20} (2013), no. 6, 1081--1090.
\bibitem[ErTz3]{ETbook}  M.~B.~Erdo\u gan and N.~Tzirakis, {\it  Dispersive Partial Differential Equations: Wellposedness and Applications}, London Mathematical Society Student Texts {  86}, Cambridge University Press, 2016. 
\bibitem[GrKo]{Gr} S. W. Graham and G. Kolesnik,  {\it Van der Corput's method of exponential sums},  London Mathematical Society Lecture Note Series, 126. Cambridge University Press, Cambridge, 1991.
\bibitem[He]{He} D. R. Heath-Brown, {\it A new k-th derivative estimate for exponential sums via Vinogradov's mean value,} Tr. Mat. Inst. Steklova 296 (2017), Analiticheskaya i Kombinatornaya Teoriya Chisel, 95--110.


\bibitem[HoVe]{HV} F.~de la Hoz and L.~Vega, {\it Vortex filament equation for a regular polygon}, Nonlinearity 27 (2014), no. 12, 3031--3057. 


\bibitem[IwKo]{Iw} H. Iwaniec and E. Kowalski, {\it Analytic number theory,}
AMS  Colloquium Publications, 53. AMS, Providence, RI, 2004. 
\bibitem[Ja]{Ja} S. Jaffard, {\it The spectrum of singularities of Riemann's function}, Revista Matem\'atic Iberoamericana 12 (1996) 441--460.
\bibitem[KaRo]{KR} L.~Kapitanski and I.~Rodnianski, {\it Does a quantum particle knows the time?,} in: {\em Emerging applications of number theory,} D. Hejhal, J. Friedman, M. C. Gutzwiller and A. M. Odlyzko, eds., IMA Volumes in Mathematics and its Applications, vol. 109, Springer Verlag, New York, 1999, pp. 355--371.
\bibitem[Ka]{katz} Y.~Katznelson, {\it An introduction to harmonic analysis. Third edition.}    Cambridge Mathematical Library. Cambridge University Press, Cambridge, 2004.

\bibitem[Kh]{khin} A.~Y.~Khinchin, {\it Continued fractions}, Translated from the 3rd Russian edition of 1961, The University of Chicago Press, 1964.
\bibitem[Le]{levy} P.~L\'evy, {\it Th\'eorie de l'addition des variables al\'eatoires,} Paris, 1937.
\bibitem[Ol]{olv} P.~J.~Olver, {\it Dispersive quantization,} Amer. Math. Monthly {  117} (2010), no. 7, 599--610.
\bibitem[OlSh]{olvernew} P.~J.~Olver and N. E. Sheils, {\it Dispersive Lamb Systems}, preprint 2017, arXiv:1710.05814v1.
\bibitem[OlTs]{olvtsa} P.~J.~Olver and E.~Tsatis, {\it Points of constancy of the periodic linearized
Korteweg–deVries equation}, preprint 2018, arXiv:1802.01213v1.
\bibitem[Os1]{O} K.~I.~Oskolkov, {\it A class of I. M. Vinogradov's series and its applications in harmonic analysis,} in: {\em ``Progress in approximation theory"}
(Tampa, FL, 1990), Springer Ser. Comput. Math. 19, Springer, New York, 1992, pp. 353--402.
\bibitem[Os2]{O2} K.~I.~Oskolkov, {\it The Schr{\oo}dinger density and the Talbot effect}. Approximation and probability, 189--219, Banach Center Publ., 72, Polish Acad. Sci. Inst. Math., Warsaw, 2006. 
\bibitem[OsCh1]{OC10} K.~I.~Oskolkov and M.~A.~Chakhkiev, {\it On the ``nondifferentiable'' Riemann function and the Schr{\oo}dinger equation}. (Russian) Tr. Mat. Inst. Steklova 269 (2010), Teoriya Funktsii i Differentsialʹnye Uravneniya, 193--203; translation in Proc. Steklov Inst. Math. 269 (2010), no. 1, 186--196.
\bibitem[OsCh2]{OC13} K.~I.~Oskolkov and M.~A.~Chakhkiev, {\it Traces of the discrete Hilbert transform with quadratic phase}. (Russian) Tr. Mat. Inst. Steklova 280 (2013), Ortogonalʹnye Ryady, Teoriya Priblizheniĭ i Smezhnye Voprosy, 255--269; translation in Proc. Steklov Inst. Math. 280 (2013), no. 1, 248--262.
  
\bibitem[Ro]{R} I.~Rodnianski, {\it Fractal solutions of the Schr\"odinger equation,} Contemp. Math. { 255} (2000), 181--187.  
\bibitem[StSh]{St}  E.~M.~Stein and R.~Shakarchi, {\it Real analysis. Measure theory, integration, and Hilbert spaces,}  Princeton University Press, Princeton, 2005; xx+402 pp
\bibitem[Tal]{talbot} H.~F.~Talbot, {\it Facts related to optical science,} No. IV, Philo. Mag. { 9} (1836), 401--407.
\bibitem[Ta1]{mtay1}  M.~Taylor, {\it Tidbits in Harmonic Analysis}, Lecture Notes, UNC, 1998.
\bibitem[Ta2]{mtay2} M.~Taylor, {\it The Schr\"odinger equation on spheres,} Pacific J. Math. {  209} (2003), 145--155.
\bibitem[Tr]{tri} H.~Triebel, {\it Theory of function spaces}, Birkh\"auser, Basel, 1983.
\bibitem[Va]{Vau} R.~C.~Vaughan, {\it The Hardy-Littlewood method},  Cambridge Tracts in Mathematics, 80. Cambridge University Press, Cambridge-New York, 1981. xi+172 pp. ISBN: 0-521-23439-5
\bibitem[Ve]{V} L.~Vega, {\it The dynamics of vortex filaments with corners,} Commun. Pure Appl. Anal. 14 (2015), no. 4, 1581--1601.
\bibitem[Wo]{Woo}  T.~D.~Wooley, {\it Nested efficient congruencing and relatives of Vinogradov’s mean value theorem}, preprint.
\end{thebibliography}
\end{document}